\documentclass[12pt]{amsart}
\usepackage{amsmath,amssymb,amsthm, amscd}
\usepackage{pifont,mathrsfs}
\usepackage{array,lastpage}
\usepackage{enumerate,xspace,pifont,shadow}
\usepackage{mathrsfs}
\usepackage [english]{babel}
\usepackage [autostyle, english = american]{csquotes}
\MakeOuterQuote{"}

\DeclareSymbolFont{AMSb}{U}{msb}{m}{n}
\DeclareMathSymbol{\Z}{\mathbin}{AMSb}{"5A}
\DeclareMathSymbol{\R}{\mathbin}{AMSb}{"52}
\DeclareMathSymbol{\N}{\mathbin}{AMSb}{"4E}
\DeclareMathSymbol{\Q}{\mathbin}{AMSb}{"51}

\newcommand{\Th}{\textup{Th}}

\newcommand{\mc}[1]{\mathcal{#1}}
\newcommand{\mf}[1]{\mathfrak{#1}}
\newcommand{\ob}[1]{\overline{#1}}
\newcommand{\abs}[1]{\left \vert #1 \right \vert}

\def\Ind{\setbox0=\hbox{$x$}\kern\wd0\hbox to 0pt{\hss$\mid$\hss}
\lower.9\ht0\hbox to 0pt{\hss$\smile$\hss}\kern\wd0}

\def\Notind{\setbox0=\hbox{$x$}\kern\wd0\hbox to 0pt{\mathchardef
\nn=12854\hss$\nn$\kern1.4\wd0\hss}\hbox to
0pt{\hss$\mid$\hss}\lower.9\ht0 \hbox to
0pt{\hss$\smile$\hss}\kern\wd0}

\newtheorem{thm}{Theorem}[section]
\newtheorem{lem}[thm]{Lemma}

\newtheorem{prop}[thm]{Proposition}

\newtheorem{fact}[thm]{Fact}
\newtheorem{ass}[thm]{Assumption}

\theoremstyle{definition}
\newtheorem{definition}[thm]{Definition}

\theoremstyle{remark}
\newtheorem{remark}[thm]{Remark}

\theoremstyle{remark}

\theoremstyle{remark}
\newtheorem{claim}[thm]{Claim}

\theoremstyle{remark}

\theoremstyle{remark}

\begin{document}
\bibliographystyle{plain}

\title{A Characterization of Strongly Dependent Ordered Abelian Groups}

\author{Alfred Dolich and John Goodrick}

\begin{abstract}We characterize all ordered Abelian groups whose first order theory in the language $\{+, <\}$ is strongly dependent. The main result of this note was obtained independently by Halevi and Hasson \cite{strong_group1} and Farr\'e \cite{strong_group2}.
\end{abstract}

\maketitle

\section{Introduction}

It is natural given a general model theoretic notion, such as stability or the absence of the independence property (see \cite{bible}), which may be interpreted as indicating that a structure or theory is ``tame'', to attempt to characterize within a general class of algebraic objects which of these objects satisfy the tameness condition.  In this note we carry out this program in the context of the tameness conditions given by strong dependence and finite dp-rank in the algebraic context of ordered Abelian groups.

Recall the following: 

\begin{definition}A theory $T$ in a language $\mc{L}$ {\em admits an ict-pattern of depth $\kappa$} (for $\kappa$ a cardinal)  if we may find a sequence $\{\varphi_i(x, \ob{y}_i) : i \in \kappa\}$ of $\mc{L}$-formulas, a model $\mf{M}$ of $T$ and parameters $\ob{a}_i^j$ for $(i,j) \in \kappa \times \omega$ in $M$ (where $\abs{\ob{a}_i^j}=\abs{\ob{y}_i}$ for all $j$) so that for each  $\eta: \kappa \to \omega$ the type \[\bigwedge_{i \in \kappa}\varphi_i(x, \ob{a}_i^{\eta(i)}) \wedge \bigwedge_{j \not= \eta{(i)}}\neg\varphi_i(x, \ob{a}_i^j)\] is consistent.  $T$ is {\em strongly dependent} if it does not admit an ict-pattern of depth $\aleph_0$ and $T$ has {\em dp-rank equal to n} if it admits an ict-pattern of depth $n$ but does not admit an ict-pattern of depth $n+1$.  $T$ is said to be {\em dp-minimal} if it has dp-rank equal to $1$.
\end{definition}

 We remark that it is a well known fact that if a theory $T$ is strongly dependent then it is NIP. See \cite{strong_dep} for a detailed discussion of NIP, strong dependence and dp-rank.  As noted above, we will primarily be concerned with ordered Abelian groups $(G,<)$ whose first order theory in the language of ordered groups is strongly dependent, providing a complete characterization of such groups.


Throughout $G$ is an ordered Abelian group in the language $\mathcal{L}_{oag}=\{+,<,0\}$.  Recall that a prime $p \in \N$ is called {\em singular} for $G$ if $[G:pG]=\infty$.  By Gurevich and Schmitt \cite{GurSchmitt}, $G$ is NIP.

We state our main result bounding the dp-rank of ordered Abelian groups using some notation which will be defined in the following section.

\begin{thm}
\label{dp_bound}
Suppose that $G$ is an ordered Abelian group considered as an $\mathcal{L}_{oag}$-structure. Then the dp-rank of of $G$ finite if and only if the following two conditions \textbf{both} hold:

\begin{enumerate}
\item $G$ has only finitely many singular primes; and
\item For any singular prime $p$, the auxiliary sort $\mathcal{S}_p$ (see below) is finite.
\end{enumerate}

Moreover, the condition that $G$ is strongly dependent is also equivalent to the conjunction of conditions (1) and (2) above.
Furthermore, when these conditions hold, the dp-rank of $G$ is bounded above by $$1 + \sum_{p \in \mathbb{P}_{sing}}  |\mathcal{S}_p |,$$ where $\mathbb{P}_{sing}$ is the set of all primes $p$ which are singular for $G$ (that is, $[G : pG] = \infty$).



\end{thm}

Theorem \ref{dp_bound} was established independently by Halevi and Hasson in their preprint ``Strongly dependent ordered abelian groups and Henselian fields'' \cite{strong_group1} as well as by Rafel Farr\'{e} in the preprint ``Strong ordered Abelian groups and dp-rank'' \cite{strong_group2}.  We also note that this characterization of strongly dependent ordered Abelian groups has already been used by Halevi and Hasson to prove that for any strongly dependent pure field $K$ and any henselian valuation $v$ on $K$, the two-sorted structure $(K, vK)$ is strongly dependent \cite{strong_group1}.

Our intention in releasing this notes is not to pre-empt either of these works but rather to provide an alternate and potentially more na\"{i}ve proof of the basic theorem.  We encourage the reader to consult either \cite{strong_group2} or \cite{strong_group1} for more definitive accounts of these results.


\section{The languages $\mathcal{L}_{eq}$ and $\mathcal{L}_2$ for quantifier elimination}

Our proof of Theorem \ref{dp_bound} is highly dependent on the work of Cluckers and Halupczok in \cite{CluckersHalup} on quantifier elimination in ordered Abelian groups thus in this section
 we will review some notation and fundamental results from  \cite{CluckersHalup} on a useful language $\mathcal{L}_{eq}$ for eliminating quantifiers in ordered abelian groups. We will not give the full details of how to define this language, since in fact the simpler language $\mathcal{L}_2$ (essentially like the langage $\mathcal{L}_{short}$ given by Jahnke, Simon, and Walsberg \cite{JSW}) will suffice for eliminating quantifiers when all the $\mathcal{S}_p$ are finite, as is always the case when the group has finite dp-rank (see Theorem \ref{sp_inf}).

\begin{definition}
For a positive integer $n$ and $a \in G \setminus n G$, $H_n(a)$ is the largest convex subgroup of $G$ such that $$a \notin H_n(a) + nG,$$ which turns out to always be a definable subgroup of $G$. If $a \in nG$, we set $H_n(a) = \{0\}$.

For $a, a' \in G$, say $a \sim_n a'$ if $H_n(a) = H_n(a')$, and let $\mathcal{S}_n$ be the imaginary sort $G / \sim_n$. The sorts $\mathcal{S}_n$ are linearly ordered by inclusion. Let $\mathfrak{s}_n : G \rightarrow \mathcal{S}_n$ be the canonical surjection.

If $\alpha \in \mathcal{S}_n$ and $\alpha = \mathfrak{s}_n(a)$, then we write ``$G_\alpha$'' as an abbreviation for $H_n(a)$.
\end{definition}

\begin{remark}
As pointed out in \cite{CluckersHalup}, for any fixed $n$, the class of subgroups $\{G_\alpha \, | \, \alpha \in \mathcal{S}_n\}$ is equal to $\{G_\alpha \, | \, \alpha \in \mathcal{S}_p, \textup{ prime, and } p | n \}$. Therefore from now on we will only consider sorts of the form $\mathcal{S}_p$ for $p \in \mathbb{P}$ where $\mathbb{P}$ is the set of all prime numbers. 
\end{remark}

We write $H \leq_{con} G$ if $H$ is a convex subgroup of $G$.

\begin{definition}
If $\alpha \in \mathcal{S}_p$ and $m $ is a positive integer, $$G^{\left[m\right]}_\alpha = \bigcap \left\{H + m G \, : \, H \leq_{con} G,\\ H \supsetneq G_\alpha \right\}.$$
If $m, m'$ are positive integers and $x, y \in G$, write $$x \equiv^{\left[m\right]}_{m, \alpha} y$$ for the relation $$x - y \in G^{\left[m'\right]}_\alpha + mG.$$
\end{definition}

The groups $G^{\left[m\right]}_\alpha$ and the relations $x \equiv^{\left[m\right]}_{m, \alpha} y$ are always definable, by the following fact (Lemma~2.4 from \cite{CluckersHalup}):

\begin{fact}
\label{G_bracket}
$$G^{\left[n\right]}_\alpha = \bigcap \left\{ G_{\alpha'} + n G \, : \, \alpha' \in \mathcal{S}_n, \, \alpha' > \alpha \right\}$$
\end{fact}

The following fact is a direct corollary of Cluckers and Halupczok's quantifier elimination, and is proved by a similar argument as in \cite{JSW}.

\begin{fact}
\label{qe_short}
Suppose that for every prime $p$, the sort $\mathcal{S}_p$ is finite. Then the complete theory of $(G; \leq, +)$ eliminates quantifiers in the extension $\mathcal{L}_2$ of $\mathcal{L}_{oag}$ which contains the following additional symbols:
\begin{enumerate}
\item Symbols for $0$ (constant) and $-$ (unary function);
\item Binary predicates $\equiv_m$ for the relation $x - y \in mG$, for each positive $m \in \N$;
\item Binary predicates $\equiv_{m, \alpha}$ for each positive $m \in \N$ and $\alpha \in \mathcal{S}_p$, denoting the relation $$x \equiv_{m, \alpha} y \Leftrightarrow x - y \in G_\alpha + mG;$$
\item Unary predicates for the (countably many) convex subgroups $G_\alpha$, where $\alpha \in \mathcal{S}_p$ for some prime $p$;
\item Constants naming a countable elementary submodel $G_0$ of $G$.
\end{enumerate}
\end{fact}

\begin{proof}
This follows by the same argument in the proof of Proposition~5.1 of \cite{JSW} showing that the quantifier elimination language $\mathcal{L}_{qe}$ from \cite{CluckersHalup} can be simplified (in the case of a group with no singular primes) to a language they call $\mathcal{L}_{short}$, but we review some details of the argument here. The language $\mathcal{L}_2$ contains all the symbols of $\mathcal{L}_{short}$ except for the unary predicates $U_{n, \overline{a}}(x)$ for the cosets of $G / nG$, but these were used in \cite{JSW} only to define the congruence relations $\equiv_m$, which we have explicitly included in our language $\mathcal{L}_2$.

The quantifier elimination language $\mathcal{L}_{qe}$ also contains symbols for the relations $\equiv^{[n]}_{m, \alpha}$, but by Fact~\ref{G_bracket}, when every $\mathcal{S}_p$ is finite these are equivalent to instances of $\equiv_{m, \alpha}$. The only remaining symbols in $\mathcal{L}_{qe}$ are relations on the auxiliary sorts $\mathcal{S}_p$, but since these sorts are all finite and $\mathcal{L}_2$ includes constants for a countable model in our language, these relations can clearly be defined in $\mathcal{L}_2$ without quantifiers.
\end{proof}

\section{Infinite $\mathcal{S}_p$ implies not strongly dependent}

In \cite{JSW} Jahnke, Simon, and Walsberg show that if $G$ has no non-singular primes then $G$ is dp-minimal.  In \cite{CKS} Chernikov, Kaplan, and Simon show that if $G$ has infinitely many singular primes then $G$ is not strongly dependent.  Thus it is obvious to conjecture that if $G$ has only finitely many singular primes then $G$ is strongly dependent (and in fact of finite dp-rank).  In this section we show that this is false and that in fact a group can have only one singular prime and nonetheless still not be strongly dependent.

  For convenience notice the following facts, which is inherent in \cite{CluckersHalup}.

\begin{fact} If $a, b \in G$ are equivalent modulo $pG$ then $H_p(a)=H_p(b)$.  In particular if $p$ is non-singular then $\mathcal{S}_p$ is finite.
\end{fact}

\begin{proof}  
This follows immediately from the definitions, since for $a$ and $b$ which are equivalent modulo $pG$ and any convex subgroup $H$ of $G$, we have that $a \in H + pG$ if and only if $b \in H + pG$.
\end{proof}

\begin{fact}\label{coord} If $H_p(a) \subset H_p(b)$ then we can find $a' \in H_p(b)$ so that $H_p(a)=H_p(a')$.
\end{fact}

\begin{proof}  Note that by definition $a \in H_p(b)+pG$.  Let $a=a'+pg$ where $a' \in H_p(b)$.  Then $a$ and $a'$ are equivalent modulo $pG$ and thus by the previous fact $H_p(a)=H_p(a')$.
\end{proof}

\begin{thm}\label{sp_inf}  Suppose that for some prime $\mc{S}_p$ is infinite then $G$ is not strongly dependent.
\end{thm}

\begin{proof}  Without loss of generality we may assume that $G$ is a sufficiently saturated model of $\Th(G)$.  Thus we may find $e_i \in G$ for $i \in \omega$ so that $H_p(e_i) \subset H_p(e_{i+1})$.  By Fact \ref{coord} we may assume that $e_i \in H_p(e_{i+1})$.  Fix $i \in \omega$,  we choose $f_{i,j} \in G$ for $j \in \omega$ so that 
$H_p(e_i) \subset H_p(f_{i,j}) \subset H_p(f_{i,j+1}) \subset H_p(e_{i+1})$ for every $j \in \omega$.  Again, we may assume that $f_{i,j} \in H_p(f_{i,j+1})$.  

Let $c_{i,j}=p^if_{i,j}$ and let $\alpha_i$ be the element of the sort $\mathcal{S}_p$ such that $G_{\alpha_i} = H_p(e_i)$.

\begin{claim}
\label{row_inconsistency}
If $j_0 \neq j_1$, then $$c_{i, j_0} \not\equiv_{p^{i+1}, \alpha_i} c_{i, j_1}.$$
\end{claim}

\begin{proof}
Without loss of generality, $j_0 < j_1$. Suppose, towards a contradiction, that these two elements are $\equiv_{p^{i+1}, \alpha_i}$-equivalent. Then there are $g \in H_p(e_i)$ and $h \in G$ such that $$c_{i, j_0} - c_{i, j_1} = p^i f_{i, j_0} - p^i f_{i, j_1} = g + p^{i+1} h.$$ Thus the element $g$ is divisible by $p^i$, and by convexity of $H_p(e_i)$ there is $g' \in H_p(e_i)$ such that $g = p^i g'$.

Now comparing with the previous displayed equation above, we can cancel out the factors of $p^i$ to conclude that $$f_{i, j_0} - f_{i, j_1} = g' + p h,$$ so $$f_{i, j_1} = -g' -ph +f_{i, j_0} \in pG + H_p(f_{i, j_1})$$ (since $j_0 < j_1$ implies that $f_{i,j_0} \in H_p(f_{i, j_1})$), but this contradicts the definition of $H_p(f_{i,j_1})$.
\end{proof}

\begin{claim}
\label{path_consistency}
For any $n \in \omega$ and any $\eta: [n] \rightarrow [n]$ (where $[n] = \{1, 2, \ldots, n\}$), the formula $$\bigwedge_{i=1}^n x \equiv_{p^{i+1}, \alpha_i} c_{i, \eta(i)}$$ is consistent, and is satisfied by the element $$a := \sum_{i=1}^n c_{i, \eta(i)}.$$
\end{claim}

\begin{proof}
This reduces to checking that if $i \in [n]$ and $j \in [n] \setminus \{i\}$, then $c_{j, \eta(j)} \in p^{i+1} G + H_p(e_i).$ But on the one hand, if $j < i$, then $$c_{j, \eta(j)} = p^j f_{j, \eta(j)} \in H_p(f_{j, \eta(j) + 1}) \subseteq H_p(e_i);$$ and on the other hand, if $j > i$, then $c_{j, \eta(j)} = p^j f_{j, \eta(j)} \in p^{i+1} G$.
\end{proof}

Now from Claims~\ref{row_inconsistency} and \ref{path_consistency} above, we conclude that there is an inp-pattern of depth $\omega$ in $G$ whose $i$-th row consists of the formulas of the form $$x \equiv_{p^{i+1}, \alpha_i} c_{i,j}$$ as $j$ varies over $\omega$.

\end{proof}

Notice that Theorem \ref{sp_inf} together with the result from \cite{CKS} mentioned above established the the ``only if'' portion of Theorem \ref{dp_bound}

\section{Bounding the dp-rank of $G$ from above}

In this section we will prove that if conditions (1) and (2) of Theorem~\ref{dp_bound} hold, then the dp-rank of $G$ is finite. At the same time, we will also establish the upper bound on the dp-rank given in Theorem~\ref{dp_bound}.

We find it convenient to work with inp-patterns rather than ict-patterns (see below for the definition). Recall that in an NIP theory, the maximal size of an inp-pattern in a single variable $x$ (the \emph{burden} of $T$) is equal to the dp-rank of $T$ \cite{onsh_usv}.

We recall some basic facts and set some notation which we will use throughout this section. As the burden of $T=Th(G)$ is equal to the dp-rank of $T$, we have that the dp-rank of $T = Th(G)$ is at least $n$ if and only if  there is an array of formulas (an \emph{inp-pattern of depth $n$}) $$\{ \varphi_i(x; \overline{a}_{ij}) : 1 \leq i \leq n, j \in \omega\}$$ such that:

\begin{enumerate}
\item For every $i \in \{1, \ldots, n\}$, there is a $k_i \in \omega$ such that $$\{\varphi_i(x; \overline{a}_{ij}) : j \in \omega\}$$ is $k_i$-inconsistent; and
\item For every $\eta : \{1, \ldots, n\} \rightarrow \omega$, the set $$\{\varphi_i(x; \overline{a}_{i, \eta(i)}) : 1 \leq i \leq n\}$$ is consistent.
\end{enumerate}

By arguments which are now standard, we can also safely assume that

\begin{enumerate}
\setcounter{enumi}{2}
\item The subindices $j$ range over all of $\Q$; and
\item The array of parameters $\overline{a}_{ij}$ is ``mutually indiscernible;'' that is, for each $i \in \{1, \ldots, n\}$, the sequence $\{ \overline{a}_{i,j} : j \in \Q\}$ is indiscernible over the set consisting of the union of all the tuples $\overline{a}_{i',j}$ such that $i' \neq i$.
\end{enumerate}

From now on, we work in an ordered Abelian group $G$ such that the sorts $\mathcal{S}_p$ (for $p$ a singular prime) are all finite, and we fix such an inp-pattern of depth $n$ in the language $\mathcal{L}_2$ described above. By quantifier elimination, we may assume that each formula $\varphi_i$ is quantifier-free.

The following is easy and already known, but we include it for convenient reference:

\begin{lem}
We may further assume that each formula $\varphi_i$ is a conjunction of literals (atomic formulas or negations of atomic formulas).
\end{lem}

\begin{proof}
Write each $\varphi_i$ as a disjunction of conjunctions of literals, say $$\varphi_i(x; \overline{y}_i) = \bigvee_{\ell=1}^{m_i} \theta_{i,\ell}(x; \overline{y}_j).$$ Then there are $\ell_1, \ldots, \ell_n$ such that $$\theta_{1,\ell_1}(x; \overline{a}_{1,0}) \wedge \ldots \wedge \theta_{n, \ell_n}(x; \overline{a}_{n, 0})$$ is consistent. Replace each formula $\varphi_i$ by $\theta_{i, \ell_i}$. The $k_i$-inconsistency of each row is clearly preserved, and the mutual indiscernibility of the parameters ensures that we also have the consistency condition we require.
\end{proof}

Now we need to consider in more detail the literals which constitute each formula $\varphi_i(x; \overline{a}_{ij})$. 

Note that we may safely assume that each literal in every formula $\varphi_i(x; \overline{a}_{ij})$ actually mentions the variable $x$. Furthermore, we may use the fact that $\mathcal{L}_2$-terms are linear functions of their variables to put every literal in every formula $\varphi_i(x; \overline{a}_{ij})$ into one of the following four types:

\bigskip

\textbf{Type (I):} $kx \equiv_{m, \alpha} t(\overline{a}_{ij})$ for some $k, m \in \N \setminus \{0\}$, $\alpha \in \mathcal{S}_p$ with $p$ a singular prime, and $\mathcal{L}_{oag}$-term $t(\overline{y})$;

\bigskip

\textbf{Type (II):} $\neg( k x \equiv_{m, \alpha} t(\overline{a}_{ij}))$ for some $k, m, \alpha,$ $p$, and $t(\overline{y})$ as above;

\bigskip

\textbf{Type (III):} Literals of the form $kx \diamond t(\overline{a}_{ij})$ where $\diamond \in \{<, >, \leq, \geq, =\}$, or of the type $k x \in G_\alpha + t(\overline{a}_{ij})$, where $t$ is a term and $\alpha \in \mathcal{S}_p$ for some prime $p$;

\bigskip

\textbf{Type (IV):} Literals of the form $kx \neq t(\overline{a}_{ij})$ or $kx \notin G_\alpha + t(\overline{a}_{ij})$, with $k, t,$ and $\alpha$ as above.

\bigskip

It is convenient to allow the parameter $\alpha$ to name the subgroup $\{0\}$ so that literals of Type (I) and (II) encompass simple congruence relations $x \equiv_m t(\overline{a}_{ij})$ and $\neg(x \equiv_m t(\overline{a}_{ij}))$. With this, we see that the division into the four types is exhaustive.

\begin{ass}
\label{minimality}
For each $i \in \{1, \ldots, n\}$, the formula $\varphi_i(x; \overline{a}_{ij})$ is a \emph{minimal} conjunction of literals, in the sense that if we were to remove any one of these literals from the conjunction, then the resulting row of formulas $$\{\varphi'_i(x; \overline{a}_{i,j}) : j \in \Q\}$$ would be consistent.
\end{ass}

\begin{prop}
\label{main_dichotomy}
Under Assumption~\ref{minimality}, each formula $\varphi_i(x, \overline{y})$ is either (a) a \textbf{single} formula $k x \equiv_{m, \alpha} t(\overline{y})$ of Type (I), or else (b) a conjunction of literals of Type (III).

\end{prop}

\begin{proof}
Call a literal $\psi(x; \overline{a}_{ij})$ occurring in $\varphi_i(x; \overline{a}_{ij})$ \emph{fixed} if it defines the same subset of $G$ even as $j$ varies, and call it \emph{variable} otherwise.


As a first observation, if $\varphi_i(x; \overline{a}_{ij})$ contains a single literal of Type (I) which is variable, then by minimality this must be the only conjunct in $\varphi_i(x; \overline{a}_{ij})$, and we are done. So we may assume that any literal in $\varphi_i(x; \overline{a}_{ij})$ of Type (I) is fixed, and our goal will be to show that in fact every literal is of Type (III).

Write $$\varphi_i(x; \overline{a}_{ij}) = \psi_1(x; \overline{a}_{ij}) \wedge \psi_2(x; \overline{a}_{ij}) \wedge \psi_3(x; \overline{a}_{ij}),$$ where:

\bigskip

 $\bullet$ $\psi_1$ is the conjunction of all fixed literals of Type (I), (II) or (IV), 
 
 \bigskip
 
$\bullet$ $\psi_2$ is the conjunction of all variable literals of Type (II) or (IV), and 
 
 \bigskip
 
$\bullet$ $\psi_3$ is the conjunction of all Type (III) literals.

\bigskip

We allow the possibility that there are no literals of one of these types, in which case the corresponding $\psi_i(x; \overline{a}_{ij})$ is equivalent to $x = x$. We will assume, towards a contradiction, that not all literals of $\varphi_i(x; \overline{a}_{ij})$ are contained in $\psi_3(x; \overline{a}_{ij})$. Thus by minimality, $\{\psi_3(x; \overline{a}_{ij}) : j \in \Q\}$ is consistent.

\begin{claim}
\label{incons}
For any $j \in \Q$, there is a finite $F \subseteq \Q \setminus \{j\}$ such that the formula $$\psi_1(x; \overline{a}_{ij}) \wedge \psi_3(x; \overline{a}_{ij})  \wedge \bigwedge_{j' \in F} \psi_2(x; \overline{a}_{ij'}) $$ is inconsistent.
\end{claim}

\begin{proof}
Recall that row $i$ is $k_i$-inconsistent.

\bigskip

\textbf{Case 1:} The convex sets defined by the $\psi_3(x; \overline{a}_{ij})$ are nested: that is, there are distinct $j, j'$ such that $\psi_3(G; \overline{a}_{ij'}) \subseteq \psi_3(G; \overline{a}_{ij})$.

\bigskip

In this case, if $j \in \Q$, we can pick distinct $j(1), \ldots, j(k_i) \in \Q$ such that $$\psi_3(G; \overline{a}_{ij(1)}) \supseteq \ldots \supseteq \psi_3(G; \overline{a}_{ij(k_i)}) \supseteq \psi_3(G; \overline{a}_{i j}),$$ and by $k_i$-inconsistency of the row of the inp-pattern, $$\psi_1(x; \overline{a}_{ij}) \wedge \psi_3(x; \overline{a}_{ij}) \wedge \bigwedge_{\ell=1}^{k_i } \psi_2(x; \overline{a}_{ij(\ell)})$$ is inconsistent. 

\bigskip

\textbf{Case 2:} The convex sets defined by the $\psi_3(x; \overline{a}_{ij})$ are not nested.

\bigskip

In this case, if $j < j' < j''$, then by the consistency of $\{ \psi_3(x; \overline{a}_{i,j} \, : \, j \in \Q\}$ and the convexity of the sets each of these formulas define, $$\psi_3(G; \overline{a}_{ij'}) \subseteq \psi_3(G; \overline{a}_{ij}) \cup \psi_3(G; \overline{a}_{ij''}).$$ Given any $j \in \Q$, choose elements $j(1) < \ldots < j(k_i) < j < j(k_i + 1) < \ldots < j(2 k_i)$. Then $$\psi_3(G; \overline{a}_{ij}) \subseteq \left( \bigcap_{\ell = 1}^{k_i} \psi_3(G; \overline{a}_{i j(\ell)}) \cup \bigcap_{\ell = k_i + 1}^{2 k_i} \psi_3(G; \overline{a}_{i j(\ell)}) \right) ,$$ so as before $$\psi_1(x; \overline{a}_{ij}) \wedge \psi_3(x; \overline{a}_{ij}) \wedge \bigwedge_{\ell = 1}^{2 k_i} \psi_2(x; \overline{a}_{i j(\ell)})$$ is inconsistent.
\end{proof}

Note that Claim~\ref{incons} implies that there must be at least one literal occurring in the conjunction $\psi_2(x; \overline{y})$.

Now suppose that $\neg \theta(x; \overline{y})$ is any literal of Type (II) or (IV) occurring in $\psi_2(x; \overline{y})$, where $\theta(x; \overline{y})$ is an atomic formula. We will establish the following Claim, which will contradict our minimality assumption on $\varphi_i(x; \overline{a}_{ij})$ and finish the proof of Proposition~\ref{main_dichotomy}:

\begin{claim}
 If $\widehat{\psi_2}(x; \overline{y})$ is the smaller conjunction obtained by removing $\neg \theta(x; \overline{y})$ from $\psi_2(x; \overline{y})$, then $$\{\psi_1(x; \overline{a}_{ij}) \wedge \widehat{\psi_2}(x; \overline{a}_{ij}) \wedge \psi_3(x; \overline{a}_{ij}) : j \in \Q\}$$ is inconsistent.
 
\end{claim}

\begin{proof}
Fix any $j \in \Q$. By Claim~\ref{incons}, there is a finite $F \subseteq \Q \setminus \{j\}$ such that $$\psi_1(x; \overline{a}_{ij}) \wedge \psi_3(x; \overline{a}_{ij})  \wedge \bigwedge_{j' \in F} \psi_2(x; \overline{a}_{ij'}) $$ is inconsistent. Now by the fact that the literals in $\psi_2(x; \overline{a}_{ij})$ are variable of Type (II) or (IV), it follows that if $j' \neq j$, then $\theta(x; \overline{a}_{ij})$ implies $\neg \theta(x; \overline{a}_{i j'})$, and hence $$\psi_1(x; \overline{a}_{ij}) \wedge \psi_3(x; \overline{a}_{ij}) \wedge \theta(x; \overline{a}_{ij}) \wedge \bigwedge_{j' \in F} \widehat{\psi_2}(x; \overline{a}_{i j'})$$ is also inconsistent. But this means that the formula $$\varphi_i(x; \overline{a}_{ij}) = \psi_1(x; \overline{a}_{ij}) \wedge \psi_2(x; \overline{a}_{ij}) \wedge \psi_3(x; \overline{a}_{ij})$$ is implied by $$\psi_1(x; \overline{a}_{ij}) \wedge \psi_3(x; \overline{a}_{ij}) \wedge \widehat{\psi_2}(x; \overline{a}_{ij}) \wedge \bigwedge_{j' \in F} \widehat{\psi_2}(x; \overline{a}_{i j'}),$$ and using this repeatedly we see that an inconsistent conjunction of $k_i$ instances of $\varphi_i(x; \overline{a}_{ij})$ yields a finite inconsistent conjunction of formulas of the form $\psi_1(x; \overline{a}_{ij}) \wedge \widehat{\psi_2}(x; \overline{a}_{ij}) \wedge \psi_3(x; \overline{a}_{ij})$, as desired.

\end{proof}

This finishes the proof of Proposition~\ref{main_dichotomy}.
\end{proof}

\begin{lem}
\label{conv_row}
There is at most one $i \in \{1, \ldots, n\}$ such that row $i$ consists of a conjunction of literals of Type (III).
\end{lem}

\begin{proof}
Since literals of Type (III) define convex sets, this follows by an elementary argument; see, for instance, the proof of Theorem 4.1 of \cite{DGL}.
\end{proof}

Next, we further simplify the literals of Type (I) which appear as rows in our inp-pattern.

\begin{lem}
\label{simplify_I}
Without loss of generality, each formula $\varphi_i(x; \overline{y})$ of Type~(I) which appears in the inp-pattern is of the form $$x \equiv_{p^\ell, \alpha} t(\overline{y})$$ for some singular prime $p$, $\ell \in \N$, and some $\mathcal{L}_{oag}$-term $t(\overline{y})$.
\end{lem}

\begin{proof}
Suppose the formulas in the $i$-th row are $k x \equiv_{m, \alpha} t(\overline{a}_{ij})$.

\begin{claim}
\label{prime_powers}
Without loss of generality, $m = p^\ell$ for some singular prime $p$.
\end{claim}

\begin{proof}
Note that if $m = m_1 m_2$ with $m_1, m_2$ relatively prime, then $$kx  \equiv_{m, \alpha} t(\overline{a}_{ij}) \Leftrightarrow \left( kx  \equiv_{m_1, \alpha} t(\overline{a}_{ij}) \wedge kx  \equiv_{m_2, \alpha} t(\overline{a}_{ij}) \right).$$ (This statement is simply a version of the Chinese remainder theorem; see Lemma~2.7 of \cite{CluckersHalup}.) Thus $kx \equiv_{m,\alpha} t(\overline{a}_{ij})$ is equivalent to a conjunction of congruences of the form $k x \equiv_{p^\ell, \alpha}$ for prime powers $p^{\ell}$. For the $i$th row to be inconsistent, there must be some such prime power $p^\ell$ such that $\{ kx \equiv_{p^\ell, \alpha} t(\overline{a}_{ij}) \, : \, j \in \Q\}$ is inconsistent; then it is clear that $p$ must be a singular prime, and that we may replace the $i$th row with these formulas.
\end{proof}

\begin{claim}
\label{prime_k}
Without loss of generality, $p$ does not divide $k$.
\end{claim}

\begin{proof}
Suppose that $p | k$. Since each formula in the inp-pattern is consistent, $t(\overline{a}_{ij}) \in p G + G_\alpha$ for every $j \in \Q$. Also, we may extend the tuples $\overline{a}_{ij}$ in the indiscernible sequence if necessary so that they include elements $a'_{ij} \in \overline{a}_{ij}$ such that $t(\overline{a}_{ij}) \in p a'_{ij} + G_\alpha$.

Then we assert that for any $x \in G$, $$k x \equiv_{p^{\ell}, \alpha} t(\overline{a}_{ij}) \Leftrightarrow (k/p) x \equiv_{p^{\ell -1}, \alpha} a'_{ij},$$ and Claim~\ref{prime_k} follows by applying this repeatedly until no factors of $p$ in $k$ remain. To see why the assertion is true, suppose on the one hand that $k x \equiv_{p^{\ell}, \alpha} t(\overline{a}_{ij})$; then $k x \equiv_{p^{\ell}, \alpha} p a'_{ij}$, and so $$p \left( (k/p) x - a'_{ij} \right) \in G_\alpha + p^\ell G.$$ So we can write $$p \left( (k/p) x - a'_{ij} \right) = g + p^\ell h$$ with $g \in G_\alpha$ and $h \in G$. Then $g$ is $p$-divisible, and furthermore $g = p g_0 $ for some $g_0 \in G_\alpha$ (by convexity of $G_\alpha$); thus $$p \left( (k/p) x - a'_{ij} \right) = p (g_0 + p^{\ell-1} h)$$ $$\Rightarrow (k/p) x - a'_{ij} = g_0 + p^{\ell-1} h,$$ and so $(k/p) x \equiv_{p^{\ell - 1}, \alpha} a'_{ij}$ as desired. Conversely, if $$(k/p) x - a'_{ij} = g + p^{\ell -1} h$$ for $g \in G_\alpha$ and $h \in G$, then multiplying by $p$ gives $$k x - t(\overline{a}_{ij}) + G_\alpha = p^\ell h + G_\alpha,$$ and so $kx \equiv_{p^\ell , \alpha} t(\overline{a}_{ij})$.
\end{proof}

Finally, we have reduced to the case of a formula $k x \equiv_{p^\ell, \alpha} t(\overline{a}_{ij})$ where $\gcd(p^\ell, k) = 1$. Pick integers $r, s$ such that $r p^\ell + s k = 1$. We claim that for any $x \in G$, $$k x \equiv_{p^\ell, \alpha} t(\overline{a}_{ij}) \Leftrightarrow x \equiv_{p^\ell , \alpha} s \cdot t(\overline{a}_{ij}),$$ completing the proof of Lemma~\ref{simplify_I}. On the one hand, if $k x \equiv_{p^\ell, \alpha} t(\overline{a}_{ij})$ and $$kx - t(\overline{a}_{ij}) = g + p^\ell h$$ with $g \in G_\alpha$ and $h \in G$, then $$skx - s \cdot t(\overline{a}_{ij}) = sg + p^\ell sh$$ $$\Rightarrow (1 - r p^\ell) x - s \cdot t(\overline{a}_{ij}) = sg + p^\ell s h$$ $$\Rightarrow x - s \cdot t(\overline{a}_{ij}) = sg + p^\ell ( sh + rx),$$ so $x \equiv_{p^\ell, \alpha} s \cdot t(\overline{a}_{ij})$. On the other hand, if $x \equiv_{p^\ell, \alpha} s \cdot t(\overline{a}_{ij})$ and $$x - s \cdot t(\overline{a}_{ij}) = g + p^\ell h$$ with $g \in G_\alpha$ and $h \in G$, then $$kx - ks \cdot t(\overline{a}_{ij}) = kg + p^\ell kh$$ $$\Rightarrow kx - (1 - r p^\ell) t(\overline{a}_{ij}) = kg + p^\ell kh$$ $$\Rightarrow kx - t(\overline{a}_{ij}) = kg + p^\ell (kh - r t(\overline{a}_{ij})),$$ so $kx \equiv_{p^\ell, \alpha} t(\overline{a}_{ij})$.

\end{proof}

\begin{lem}
\label{Sp}
Suppose that two different rows of the inp-pattern, Row $i$ and Row $i'$, consist of Type~(I) formulas $$x \equiv_{p^\ell, \alpha} t(\overline{a}_{ij})$$ and $$x \equiv_{p^{\ell'}, \alpha'} t'(\overline{a}_{i'j})$$ respectively, with the same singular prime $p$.

Then if $G_\alpha \subseteq G_{\alpha'}$, there is some $\beta \in \mathcal{S}_p$ such that $\alpha \leq \beta < \alpha'$.
\end{lem}

\begin{proof}
First note that $\ell < \ell'$, since if $\ell' \leq \ell$, we would have $$(G_\alpha + p^\ell G) \subseteq (G_{\alpha'} + p^{\ell'} G)$$ and it wold be impossible to form two rows of an inp-pattern with the relations $\equiv_{p^\ell, \alpha}$ and $\equiv_{p^{\ell'}, \alpha}$.

Now pick $c \in G$ such that $$c \equiv_{p^\ell, \alpha} t(\overline{a}_{i,0}) \wedge c \equiv_{p^{\ell'}, \alpha'} t'(\overline{a}_{i',0}),$$ and pick $d \in G$ such that $$d \equiv_{p^\ell, \alpha} t(\overline{a}_{i,1}) \wedge d \equiv_{p^{\ell'}, \alpha'} t'(\overline{a}_{i',0}).$$ Then $$c - d \in (G_{\alpha'} + p^{\ell'} G) \setminus (G_\alpha + p^\ell G) \subseteq (G_{\alpha'} + p^{\ell} G) \setminus (G_\alpha + p^\ell G),$$ and thus $$G_\alpha \subseteq H_{p^\ell}(c-d) \subsetneq G_{\alpha'}.$$ But $H_{p^\ell} (c-d) = H_p(p^{\ell-1} (c-d))$ is a subgroup named by a sort in $\mathcal{S}_p$, so we are finished.
\end{proof}

\textit{Proof of Theorem~\ref{dp_bound}:} Suppose that the set $\mathbb{P}_{sing}$ of singular primes is finite and that $\mathcal{S}_p$ is finite for each $p \in \mathbb{P}_{sing}$, and that we have an inp-pattern of depth $n$ in a single variable $x$ satisfying all of the assumptions above.

Then at most one row consists of a conjunction of Type~(III) formulas, and all other rows consist of single Type~(I) formulas of the form $x \equiv_{p^\ell, \alpha} t(\overline{a}_{ij})$ for some $p \in \mathbb{P}_{sing}$. By Lemma~\ref{Sp}, for each singular prime $p$, there are at most $|\mathcal{S}_p| $ rows in our inp-pattern. Therefore the total depth of the inp-pattern is at most $$1 + \sum_{p \in \mathcal{P}_{sing}}  |\mathcal{S}_p| ,$$ and in particular the dp-rank of $G$ is finite. $\square$








\section{Examples}

\subsection{Optimality of the upper bound in Theorem~\ref{dp_bound}}

For any prime $p$, let $$\Q_{(p)} = \{ \frac{a}{b} \, : \, a, b \in \Z, \, b \neq 0, \, \textup{gcd}(b, p) = 1\}.$$ Fix some countably infinite subset $B \subseteq \R$ such that $1 \in B$ and the elements of $B$ are linearly independent over $\Q_{(p)}$, and let $G_p$ be the ordered subgroup of $\R$ consisting of all finite sums $a_1 b_1 + \ldots + a_k b_k$ such that $a_i \in \Q_{(p)}$ and $b_i \in B$. The essential properties of $G_p$ are that it is an Archimedean ordered abelian group, $[G_p : p G_p] = \infty$, and for any prime $q \neq p$, $q G_p = G_p$.

For $j \in \omega \setminus \{0\}$, we let $G_p^j$ stand for the direct sum of $j$ copies of $G_p$, ordered lexicographically. One can readily check that $G_p^j$ has $p$ as its unique singular prime.  One may also check that and the dp-rank of $G_p^j$ (in $\mathcal{L}_{oag}$) is $j$, although this is not necessary for the ensuing proposition and in fact follows from its proof.

\begin{prop}
For any finite sequence of elements $k_0, \ldots, k_{m-1} \in \omega \setminus \{0\}$, let $p_0, \ldots, p_{m-1}$ denote the first $m$ prime numbers and let $$G = \Q \oplus \bigoplus_{i=0}^{m-1} G_{p_i}^{k_i},$$ ordered lexicographically (so that the first coordinate in $\Q$ takes precedence).  Then the singular primes for $G$ are $p_0, \dots, p_{m-1}$, for each such $p_i$ the cardinality of $\mc{S}_{p_i}$ is $k_i$, and the dp-rank of $G$  is $1 + \sum_{i=0}^{m-1} k_i$.
\end{prop}

\begin{proof}
The fact that $p_0, \dots, p_{m-1}$ are the singular primes for $G$ is immediate.   On the one hand, for any $i < m$, if we let $$H_i = \bigoplus_{j > i} G_{p_j}^{k_j},$$ then the sort $\mathcal{S}_{p_i}$ consists of names for the convex subgroups $$0, G_{p_i} \oplus H_i, G^2_{p_i} \oplus H_i, \ldots, G^{k_i-1}_{p_i} \oplus H_i$$ (as can be checked by the definition of the groups $H_{p_i}(a)$ as a simple exercise); thus $|S_{p_i}| = k_i$, and since $p_0, \ldots, p_{m-1}$ are all the singular primes of $G$, the fact that the dp-rank of $G$ is less than or equal to $1 + \sum_{i=0}^{m-1} k_i$ follows from Theorem~\ref{dp_bound}.
 
On the other hand, to get the opposite rank inequality, we just need to exhibit an inp-pattern of depth $1 + \sum_{i=0}^{m-1} k_i$ in $G$. To accomplish this, for any $i \in \{ 0, \ldots, m-1\}$ and any $j \in \{0, \ldots, k_i - 1\}$, pick elements $\{c_{i,j,k} \, : \, k \in \omega\} \subseteq p_i^j G_{p_i}$ which represent distinct cosets of $p_i^{j+1} G_{p_i}$ and let $e_{i,j,k} \in G$ be the element whose coordinate in the $j$th copy of $G_{p_i}$ (counting from the right) is $c_{i,j,k}$ and all of whose other coordinates are equal to $0$.

Finally, we can construct an inp-pattern as follows: for each $i \in \{ 0, \ldots, m-1\}$ and each $j \in \{0, \ldots, k_i - 1\}$, construct a row of formulas $$\varphi_{i,j}(x; e_{i,j,k}) := x \equiv_{p_i^{j+1} \alpha_{i,j}} e_{i,j,k}$$ where $\alpha_{i,j}$ is an element in the sort $\mathcal{S}_{p_i}$ representing the convex subgroup $G^j_{p_i} \oplus H_i$, unless $j=0$ in which case we let $\alpha_{i,j}$ be a name for the trivial subgroup $\{0\}$. The final row in the inp-pattern will consist of pairwise disjoint intervals $a_k < x < b_k$ (for $k \in \omega$) constructed by first picking elements $a^0_0 < b^0_0 < a^0_1 < b^0_1 < \ldots$ in $\Q$ and then letting $a_k, b_k$ be elements whose $\Q$-coordinate is $a^0_k$ (or $b^0_k$, respectively) and all of whose other coordinates are equal to zero.

It is immediate that each row of the pattern described above is $2$-inconsistent, and all that is left is to explain why, given any $k \in \omega$ and any choice of values $k(i,j) \in \omega$ for each $(i,j)$ with $i \in \{0, \ldots, m-1\}$ and $j \in \{0, \ldots, k_i - 1\}$, there is an element $d \in (a_k, b_k)$ which satisfies $$d \equiv_{p_i^{j+1}, \alpha_{i,j}} e_{i, j, k(i,j)}$$ for all permissible pairs $(i,j)$. For this, we may pick $c_k \in \Q$ such that $a_k < c_k < b_k$, let $d_k \in G$ be such that its $\Q$-coordinate is $c_k$ and all of its other coordinates are $0$, and let $$d = d_k + \sum_{i < m, j < k_i} e_{i,j, k(i,j)}.$$ We leave it as an exercise to the reader to verify that this element $d$ satisfies all the required formulas.

\end{proof}

\bibliography{modelth}

\end{document}